\documentclass{amsproc}
\usepackage{hyperref}
\hypersetup{
    colorlinks=false, %set true if you want colored links
    linktoc=all,     %set to all if you want both sections and subsections linked
    linkcolor=blue,
        citecolor=red,
    filecolor=black,
    urlcolor=green%choose some color if you want links to stand out
}
\usepackage[all]{hypcap}
\usepackage{amssymb}
\usepackage[shortlabels]{enumitem}
\usepackage{bm}
\newcommand{\ga}{\alpha}
\newcommand{\gb}{\beta}
\newcommand{\gga}{\gamma}
\newcommand{\gd}{\delta}
\newcommand{\gep}{\epsilon}

\newcommand{\gth}{\theta}

\newcommand{\gl}{\lambda}
\newcommand{\gm}{\mu}
\newcommand{\gn}{\nu}
\newcommand{\gx}{\xi}

\newcommand{\gp}{\pi}

\newcommand{\gt}{\tau}

\newcommand{\gf}{\phi}

\newcommand{\Gd}{\Delta}

\newcommand{\Gs}{\Sigma}

\newcommand{\subs}{\subset}
\newcommand{\sups}{\supset}

\newcommand{\bs}{\backslash}

\newcommand{\nin}{\notin}

\newcommand{\ol}{\overline}

\newcommand{\equ}[1]{%
\begin{equation*}
#1
\end{equation*}
}
\newcommand{\equa}[1]{%
\begin{equation*}
\begin{aligned}
#1
\end{aligned}
\end{equation*}
}

\newcommand{\mbb}{\mathbb}

\newcommand{\mcl}{\mathcal}

\newcommand{\ul}{\underline}

\newcommand{\us}{\underset}
\newcommand{\os}{\overset}

\newcommand{\lra}{\longrightarrow}

\newcommand{\Z}{\mbb Z}

\newcommand{\Ra}{\Rightarrow}

\newcommand{\es}{\emptyset}

\usepackage{seqsplit}

\theoremstyle{plain}
\newtheorem{theorem}[equation]{Theorem}
\newtheorem*{theorem*}{Theorem}
\newtheorem*{conjecture*}{Conjecture}

\newtheorem*{prop*}{Proposition}
\newtheorem{lemma}[equation]{Lemma}
\newtheorem*{lemma*}{Lemma}

\newtheorem*{cor*}{Corollary}

\newtheorem{claim}[equation]{Claim}
\newtheorem*{fq*}{Main question}

\newtheorem{obs}[equation]{Observation}
\theoremstyle{definition}
\newtheorem{defn}[equation]{Definition}
\newtheorem*{defn*}{Definition}

\newtheorem*{remark*}{Remark}
\newtheorem{example}[equation]{Example}
\newtheorem*{notation*}{Notation}

\title[Approximation of $4$-Sets]{Approximation of Quadrilaterals by Rational Quadrilaterals in the Plane}

\author{C.P. ANIL KUMAR}

\address{Indian Statistical Institute , 8th Mile Mysore Road,
RVCE Post, Bangalore, Karnataka 560059, India}

\keywords{Rational Triangles and Quadrilaterals, Rational
Approximability of polygons, Rational Points on quartic Curves,
Elliptic Curves, Torsion Points, Rational Points on Varieties and
their Density}

\subjclass[2010]{Mathematics Subject Classifications: 14G05, 11J17}
\begin{document}

\maketitle

\begin{abstract}
Many questions about triangles and quadrilaterals with rational
sides, diagonals and areas can be reduced to solving certain
Diophantine equations. We look at a number of such questions
including the question of approximating arbitrary triangles and
quadrilaterals by those with rational sides, diagonals and areas. We
transform these problems into questions on the existence of
infinitely many rational solutions on a two parameter family of
quartic curves. This is further transformed to a two parameter
family of elliptic curves to deduce our main result concerning
density of points on a line which are at a rational distance from
three collinear points (Theorem~\ref{theorem:DensityLineAOC}). We
deduce from this a new proof of density of rational quadrilaterals
in the space of all quadrilaterals
(Theorem~\ref{th:RationalQuadrilateral}). The other main result
(Theorem~\ref{Th:RationalTriangle}) of this article is on the
density of rational triangles which is related to analyzing rational
points on the unit circle. Interestingly, this enables us to deduce
that parallelograms with rational sides and area are dense in the
class of all parallelograms.

We also give a criterion for density of certain sets in topological
spaces using local product structure and prove the density
Theorem~\ref{Theorem:DensityFibrations} in the appendix section. An
application of this proves the density of rational points as stated
in Theorem~\ref{theorem:RationalDensity}.
\end{abstract}

\section{Introduction}
Throughout this article, we call a polygon $\mathit{rational}$ if
its sides, diagonals and area are all rational numbers. Interest in
the theory of rational triangles goes back to the time of Leonhard
Euler. Euler found formulae expressing proportions of the sides of a
rational right-angled triangle and a general rational triangle. For
the latter, he proved:

\begin{theorem}[Euler~\cite{Euler2}]
The sides of a general rational triangle $\Gd ABC$ with sides
$AB=c,BC=a,AC=b$ with rational area satisfy the proportion
\begin{equation*}
a:b:c = \frac{r^2+s^2}{rs}:\frac{(ps\pm rq)(pr \mp
qs)}{pqrs}:\frac{p^2+q^2}{pq}
\end{equation*}
for some integers $p > q,r \geq s$.
\end{theorem}

H.F. Blichfeldt ~\cite{Blichfeldt} and D.N. Lehmer~\cite{Lehmer}
have independently derived formulae for the sides of a rational
triangle. Lehmer also characterized integral triangles
i.e. triangles with integer sides and area in the plane in
~\cite{Lehmer}.

E.E. Kummer~\cite{Kummer} obtained a characterization of rational
quadrilaterals in the plane. He reduced the problem of finding
rational quadrilaterals to the problem of finding rational solutions
to the equation.
\begin{equation}
\label{eq:symmetric} \frac{(\gx+c)^2-1}{2\gx}.\frac{(x-c)^2-1}{2x} =
\frac{(\gn-c)^2-1}{2\gn}.\frac{(y+c)^2-1}{2y}
\end{equation}
in rationals $\gx,\gn,x,y,c$ with $|c| < 1$.

L.E. Dickson~\cite{Dickson} also derived expressions for rational
quadrilaterals similar to Kummer's. In the conclusion of his
paper~\cite{Dickson} he mentions that some questions about triangles
and quadrilaterals reduce to deciding whether certain quartic
functions can be written in terms of rational squares.

In what follows, we say that a polygon of $n$ sides is rationally
approximable if there are rational polygons of $n$ sides whose
vertices are arbitrarily close to the given one. I.J. Shoenberg
posed the general question:
\begin{center}
$\mathit{Is\ every\ n-sided\ polygon\ rationally\ approximable?}$
\end{center}

A.S. Besicovitch~\cite{Besicovitch} answered this question for the
special cases of right-angled triangles and parallelograms.

D.E. Daykin ~\cite{Daykin} answered Shoenberg's question
affirmatively in the class of quadrilaterals, parallelograms, and
some classes of hexagons. We can deduce a new proof of rational
approximability for quadrilaterals as a consequence of our
Theorem~\ref{theorem:DensityLineAOC}.

Open questions regarding integral and rational distances and
rational approximation have attracted many other mathematicians such
as John Isbell, John Leech, Harborth, Kemnitz, Richard Guy, N.H.
Anning, Paul Erd{\"o}s, J.H.J Almering,  T.K. Sheng and T.G. Berry.
(See the list of references).

\vskip 3mm

\subsection{The Two Main Results}

\noindent In this paper, we prove two results. To this end we need
some definitions. Let us call a set $\{A,B,C\}$ of three points in
the plane a rational $3$-set if the lengths $AB,AC,BC$ are rational.
Call a set $\{A,B,C,D \}$ of four points on the plane a rational
$4$-set, if all the six distances are rational.  If $K \subs
\mbb{R}^2$ is a compact/finite set and $L \subs \mbb{R}^2$ is a
closed set then we define the distance $d(K,L)=min\{d(x,y) \in
\mbb{R} \mid x \in K,y \in L\}$. We also define for any two finite
$n-$sets $K_1=\{A_1,A_2,\ldots,A_n\},
K_2=\{A_1',A_2',A_3',\ldots,A_n'\} \subs \mbb{R}^2$, equipped with a
bijection between $K_1,K_2$, given by $A_i \mapsto A_i'$, the
distance $D(K_1,K_2) = max\{d(A_i,A_i') \in \mbb{R} \mid 1 \leq i
\leq n\}$. As such this definition depends on the bijection $A_i \mapsto A_i'$. 
However see section~\ref{sec:Definitions} definition~\ref{Definition:Bijection}
as we require the distance $D(K_1,K_2)$ when $d(A_i,A_i')$ is very small for 
$1 \leq i \leq n$. Now we are ready to state our two results.

The first one is on the density of rational triangles in the space
of triangles in the plane.
\begin{theorem}
\label{Th:RationalTriangle} Let $X=\{A,B,C\}$ represent three
vertices of a triangle $\Gd ABC$ in the Euclidean plane. Then given
$\gep > 0$ there exists a rational $3$-set $X_{\gep}$ in the plane
such that the points in $X_{\gep}$ form a rational triangle with
rational area and $D(X,X_{\gep}) < \gep$.

More precisely if $BC$ is the largest side of the triangle $\Gd ABC$
then we can choose the $3$-set $X_{\gep}=\{A,B',C'\}$ to also
contain the vertex $A$ opposite to the side $BC$ and make the side
$BC \parallel B'C'$.
\end{theorem}

Using this main result on rational approximability of triangles we
also deduce the analogous result for parallelograms.

The second main result addresses the density of points on a line
which are at a rational distance from three collinear points.

\begin{theorem}
\label{theorem:DensityLineAOC} Let $A,O,C$ be three distinct,
collinear points in the plane with point $O$ on the line segment
$AC$ such that the $3$-set $\{A,O,C\}$ is a rational set. Let $L$ be
a line passing through $O$ such that the sine and cosine of the
angle between $L$ and $AC$ are rational. There exist finite sets
$F_{\measuredangle} \subs [0,\gp]$ and $F_{ratio} \subs \mbb{R}$
such that
\begin{itemize}
\item For a fixed angle between $L$ and $AC$ not in the finite set
$F_{\measuredangle}$, the set of all points $B$ on $L$ such that the
four set $\{A,B,C,O\}$ is a rational set is dense, except for a
finitely many values of the ratio $\frac{AO}{OC}$.
\item For a fixed ratio $\frac{AO}{OC}$ not in the
finite set $F_{ratio}$, the set of all points $B$ on $L$ such that
the four set $\{A,B,C,O\}$ is a rational set is dense, except for a
finitely many choices of the angles between $L$ and $AC$.
\end{itemize}
\end{theorem}

For the proof we use quartic curves as well as a family of cubic
curves. From this Theorem~\ref{theorem:DensityLineAOC}, we deduce
rational approximability of general quadrilaterals.

In $1960$  L.J. Mordell~\cite{Mordell} proved that every
quadrilateral in the plane is approximable by quadrilaterals with
rational sides and diagonals (with no condition on the area). He used
Nagell's theorem on integral points and torsion points on cubic
curves.

\subsection{The Density Result}
We also prove the following useful topological density result and as
a consequence we prove density of rational points as stated in
Theorem~\ref{theorem:RationalDensity}.

For this purpose we introduce a definition
\begin{defn}[Local Product Structure]
Let $X,Y$ be topological spaces and $f:X \lra Y$ be a surjective
continuous map. Let $x_0 \in X$ and $f(x_0)=y_0$. Let $F_{x_0}$ be
any topological space. Suppose there exists open sets $O \subs X,U
\subs F_{x_0},V \subs Y$ such that $x_0 \in O,y_0 \in V$ and $O
\cong_{\psi} U \times V$ and such that the following diagram
commutes.

\equ{\big(O \us{f}{\lra} V\big) = \big(O \us{\psi}{\lra} U \times V
\us{\gp_2}{\lra} V\big).}

Then we say that $X$ has the local product structure property at the
point $x_0$ with respect to $f,F_{x_0}$.
\end{defn}

Now we state the theorem.

\begin{theorem}
\label{Theorem:DensityFibrations} Let $X_0 \us{f_1}{\lra} X_1
\us{f_2}{\lra} X_2 \us{f_3}{\lra} X_3 \us{f_4}{\lra} \ldots
\us{f_n}{\lra} X_n$ be a sequence of surjective continuous maps of
topological spaces such that the local product structure property is
satisfied on a dense set $Z_i$ of $X_i$ with respect to the map
$f_{i+1}$ for $i=0,\ldots,n-1$. Then if $B \subs X_n$ is dense then
any fibre-wise dense set in the preimage of $B$ in each $X_i$ is
dense in $X_i$ for all $0 \leq i \leq n-1$.
\end{theorem}

We end this introduction with another question which, similar to
Shoenberg's question for $n \geq 5$, is still open~\cite{Berry}.

\begin{itemize}

\item   $\mathit{Does\ there\ exist\ a\ point\ in\ the\ plane\ at\ a\ rational\ distance\ from}$\\
        $\mathit{each\ of\ the\ corners\ of\ a\ unit\ square?}$
\end{itemize}

\section{Definitions}
\label{sec:Definitions}
In this article we use the following definitions.
\begin{defn}
\label{RationalSet} Let $X \subs \mbb{R}^2$ be a subset of the
Euclidean plane. The distance set is defined as $\Gd(X) = \{r \in
\mbb{R}\mid \text{ there exists } p,q \in X \text{ with }
r=d(p,q)\}$. We say that the set $X$ is rational if $\Gd(X) \subs
\mbb{Q}$.
\end{defn}
\begin{defn}
~\\
\vspace{-0.5cm}
\label{Definition:Bijection}
\begin{itemize}
\item A finite $n-$subset $X=\{A_1,A_2,\ldots,A_n\} \subs \mbb{R}^2$ is said to be
rational approximable if given $\gep > 0$ there exists a
rational $n-$subset $X_{\gep}=\{A_1',A_2',\ldots$\\
$,A_n'\} \subs \mbb{R}^2$ with a bijection
$A_i \to A_i'$ such that $max\{d(A_i,A_i')\mid i=1,\ldots,n\}=D(X,X_{\gep})<\gep$.
Here in the case of rational approximability for a given finite set $X$ and $\gep$
small enough depending on $X$, the bijection between $X$ and $X_{\gep}$ is unique
if such a set $X_{\gep}$ exists.
\item A polygon of $n$ sides is said to be rational, if
its sides, diagonals and area are all rational.
\end{itemize}
\end{defn}

\section{Rational Approximability of Triangles and Parallelograms}

In this section, we prove that triangles are rationally
approximable. Towards that, we quote the following
Lemmas~\ref{lemma:RationalPoints}, and ~\ref{lemma:Density}.

\begin{lemma}
\label{lemma:RationalPoints}
~\\
Let $C$ be a circle centered at the origin in $\mbb{R}^2$ whose
radius is rational. Then the set $\mbb{P} = \{(x,y) \in C \mid x,y
\in \mbb{Q}\}$ of points with rational coordinates on the circle is
dense in $C$.
\end{lemma}

The proof is straightforward. See the proof given by Paul D. Humke and Lawrence L.
Krajewski~\cite{HumkeLawrence} for a characterization of circles in
the plane whose rational points are dense in their respective
circles.

The following lemma addresses the density question for angles.

\begin{lemma}
\label{lemma:Density}
~\\
\begin{enumerate}
\item Let $\mbb{Q}_{tan} = \{\gth \in \mbb{R} \mid tan(\gth) \text{ is rational or undefined }\}$.\\
\item Let $\mbb{Q}_{tan2} = \{\gth \in \mbb{R} \mid tan(\gth) = \frac{q}{p},gcd(p,q) = 1,p^2+q^2
\text{ is a square or }\\ tan(\gth) \text{ is undefined }\}$.\\
\item Let $C_{\mbb{Q}} = \{\gth \in \mbb{R} \mid Cos(\gth),Sin(\gth) \text{ are rational }\}$.\\
\end{enumerate}
Then
\begin{enumerate}[label=(\roman*)]
\item Then the set $\mbb{Q}_{tan2}$ is dense in $\mbb{R}$\\
\item $\mbb{Q}_{tan2} \subs \mbb{Q}_{tan}$ and $\mbb{Q}_{tan2} = 2\mbb{Q}_{tan} = C_{\mbb{Q}}$.\\
\item $\mbb{Q}_{tan},\mbb{Q}_{tan2}$ are additive subgroups of $\mbb{R}$
\end{enumerate}
\end{lemma}

\begin{proof}
First we observe that for every integer $k \in \Z$ the function
\equ{Tan_k:(\gp k - \frac{\gp}{2},\gp k+ \frac{\gp}{2}) \lra
\mbb{R}, \gth \mapsto Tan(\gth)} is a homeomorphism. Hence the set
$\mbb{Q}_{tan}$ is dense in $\mbb{R}$.

Now we use some elementary geometry. Let $C$ be a circle with center
$O$ of unit radius. Let $A,B$ be two points on the circle such that
the $arc AB$ subtends an angle $2\gth$ at the center. Extend $OA$ to
meet the circle again at $P$. Then the $\measuredangle APB = \gth$.

We prove $(ii)$ first.

Now we prove $\mbb{Q}_{tan2}=C_{\mbb{Q}}$. Let $\gth \in
C_{\mbb{Q}}$; then $Cos(\gth)=\frac rs, Sin(\gth)= \frac uv$ for some
relatively prime integers $r,s$ and $u,v$. So we have $\frac
{r^2v^2+u^2s^2}{s^2v^2}=1$ i.e. $r^2v^2+u^2s^2=s^2v^2$. If
$Cos(\gth)=0$ then $\gth \in \mbb{Q}_{tan2}$. Let $Cos(\gth) \neq
0$. Now we observe that $Tan(\gth)$ is rational and if for some
$q,p$ relatively prime integers $\frac qp=Tan(\gth) = \frac
{us}{rv}$. Then there exists an integer $t$ such that $us=tq$ and
$rv=tp$ so $t^2(p^2+q^2)=r^2v^2+u^2s^2=s^2v^2$. So $t^2 \mid s^2v^2
\Ra t \mid sv$ and $p^2+q^2= \big(\frac {sv}{t}\big)^2$ a perfect
square. So $\gth \in \mbb{Q}_{tan2}$. The converse is also clear;
i.e. if $\gth \in \mbb{Q}_{tan2}$ then $Cos(\gth),Sin(\gth)$ are
rational.

Now we prove $\mbb{Q}_{tan2}=2\mbb{Q}_{tan}$. Let $\gth \in
\mbb{Q}_{tan}$ and if $Tan(\gth)$ is undefined then $\gth$ is an odd
multiple of $\frac {\gp}{2}$. So $2\gth$ is an integer multiple of
$\gp$. So $Tan(2\gth)=0$ and $2\gth \in \mbb{Q}_{tan2}$. If
$Tan(\gth)=0$ then $Tan(2\gth)=0$ so $2\gth \in \mbb{Q}_{tan2}$. If
$Tan(\gth)=\frac qp$ with $gcd(q,p)=1$ then
$Tan(2\gth)=\frac{2Tan(\gth)}{1-Tan^2(\gth)}=\frac{2pq}{p^2-q^2}$.
We observe that $(p^2-q^2)^2+4p^2q^2=(p^2+q^2)^2$ a perfect square.
Hence if $Tan(2\gth) = \frac uv$ with $gcd(u,v)=1$ then also
$u^2+v^2$ is a perfect square because there exists an integer $t$
such that $2pq=tu,p^2-q^2=tv$. So $2\gth \in \mbb{Q}_{tan2}$.
Conversely it is also clear that if $2\gth \in \mbb{Q}_{tan2}$ then
$Tan(\gth)$ is rational. i.e. $\gth \in \mbb{Q}_{tan}$.

Now we prove $(iii)$. We observe that
$Tan(0)=0,Tan(-\gth)=-Tan(\gth)$ and if $\gth_1+\gth_2 \neq
(2k+1)\frac{\gp}{2}$ for some $k \in \Z$ then
$Tan(\gth_1+\gth_2)=\frac{Tan(\gth_1)+Tan(\gth_2)}{1-Tan(\gth_1)Tan(\gth_2)}$.
So $\mbb{Q}_{tan}$ is an additive subgroup. Hence
$\mbb{Q}_{tan2}=2\mbb{Q}_{tan}$ is also an additive subgroup.

Now to prove $(i)$ we observe that any finite index additive
subgroup of a dense additive subgroup of reals is also dense in
reals.
\end{proof}

We note that for a right angled triangle with rational sides, the
area is rational. From the lemma above, we deduce the following
density theorem for right angled triangles which we mention below
without proof as it is straightforward.

\begin{theorem} Let $X=\{A,B,C\}$ represent three vertices of a right-angled
triangle $\Gd ABC$ in the Euclidean plane. Then given $\gep > 0$
there exists a $3-$set $X_{\gep}$ in the plane such that the points
in $X_{\gep}$ form a rational right-angled triangle and
$D(X,X_{\gep}) < \gep$. In fact we can choose $X_{\gep}$ such that
it has any one of the points of $X$ in common.
\end{theorem}

The general case of triangles is also a straightforward consequence
of the right-angled triangles case. Now we prove
Theorem~\ref{Th:RationalTriangle} here.
\begin{proof}
Given the triangle $\Gd ABC$ in the plane, let $a$ be a largest side
among $a,b,c$. Drop a perpendicular $AD$ from the vertex $A$ to the
opposite side $BC$ with intersection point $D = AD \cap BC$. Now
$\measuredangle BAD = \ga, \measuredangle CAD = \gb$. Choose a point
$D'$ on $AD$ such that $AD'$ is rational and $d(D',D) < \gd$. Choose
by Lemma~\ref{lemma:Density}, $\ga_1,\gb_1$ such that $0 <
\ga_1,\gb_1< \frac{\gp}{2}$ and
$Cos(\ga_1),Sin(\ga_1),Cos(\gb_1),Sin(\gb_1)$ are rational and
$d(\ga,\ga_1) < \gd,d(\gb,\gb_1) < \gd$. Consider the right-angled
triangles $\Gd AD'B'$ and $\Gd AD'C'$, both having right-angles at
the vertex $D'$. Hence the line $B'C'$ is parallel to $BC$ and $a'
\os{def^n}{=} B'C' = AD'(Tan(\ga_1) + Tan(\gb_1))$ which is
rational. We also observe that $AB' = AD'Sec(\ga_1), AC' =
AD'Sec(\gb_1)$ which are rational. Finally the area of the triangle
$\Gd AB'C'$ is $\frac{1}{2}a'(AD')$ which is rational. Next choose
$\gd$ such that $D(X,\{A,B',C'\}) < \gep$ and take $X_{\gep} =
\{A,B',C'\}$. Here again we observe that the vertex $A$ is unchanged
in the approximant $X_{\gep}$ and $BC \parallel B'C'$.

In the case when the three points lie on a line then the proof is straight forward.
\end{proof}
We now use the above theorem on triangles to deduce the analogous
result in the class of parallelograms.

\begin{theorem}
Let $X=\{A,B,C,D\}$ represent the vertices of a parallelogram
$\square ABCD$ in the Euclidean plane. Then given $\gep > 0$ there
exists a rational $4$-set $X_{\gep}$ in the plane such that the
points in $X_{\gep}$ form a rational parallelogram with rational
area and $D(X,X_{\gep}) < \gep$.
\end{theorem}
\begin{proof}
Let $AC,BD$ be the diagonals of the parallelogram such that $AC \geq
BD$. Then $AC$ is the largest side of the congruent triangles $\Gd
ABC$ and $\Gd ADC$ because $\measuredangle ABC=\measuredangle ADC$
is obtuse or just right. Using Theorem~\ref{Th:RationalTriangle} we
get an approximant $\Gd A'B'C'$ such that $B'=B$ and $A'C' \parallel
AC$. Now we parallel translate $\Gd A'B'C'$ so that line determined
by the line segment $A'C'$ coincides with that of $AC$. Then we
complete this to a parallelogram $\square A'B'C'D'$ with $A'C'$ and
$B'D'$ as diaagonals. In this procedure while approximating we make
sure $D(\{A,B,C\},\{A',B',C'\}) < \gep$. So that by symmetry of
parallelograms we obtain an $\gep$-approximant rational
parallelogram $\square A'B'C'D'$ to parallelogram $\square ABCD$.
Moreover it has rational area as the area is twice the area of the
rational triangle $\Gd A'B'C'$.
\end{proof}

\section{Rational Points on a Hyperbola}

The proof of our general results on quadrilaterals in section
$\ref{sec:MainTheoremOnDensity}$ requires some analysis of rational
points on hyperbolae. Indeed, we prove:

\begin{theorem}
\label{lemma:DensityLinePQ}
~\\
Let $P,Q$ be two points in the plane at a rational distance from
each other. Let $L$ be a line passing through $P$ such that the
cosine of the angle between $L$ and $PQ$ is rational. Then the set
of all points on $L$ which are at a rational distance from $P$ and
$Q$ are dense in $L$. Conversely if there exists a point on $L$
which is at a rational distance from $P$ and $Q$ then the cosine of
the angle between $L$ and $PQ$ is rational.
\end{theorem}
\begin{proof}
Assume without loss of generality the line $PQ$ represents $x$-axis
with $P$ as the origin and $Q$ is at a rational distance $r$ from
the origin. Let the equation of the line $L$ be $y=Tan(\gth)x$ where
$Cos(\gth)$ is rational which implies $Sin^2(\gth),Cos^2(\gth)$,
$Tan^2(\gth)$ are rational if $\gth \neq (2k+1)\frac{\gp}{2}, k \in
\Z$ with $Tan(\gth),Sin(\gth)$ need not be rational. Let $R$ be a
point on $L$ at a distance $q$ from the origin and at a distance $p$
from $Q$ i.e. $\Gd PQR$ is a triangle with $PQ = r,QR = p,RP = q$.
We see by computing distances, if we set $s^2=p^2-r^2Sin^2(\gth)$
then $q = rCos(\gth) + s$. If $s$ is rational then $q$ is rational
since $r,Cos(\gth)$ are rational.

To proceed with the proof, we need the following observation.
\begin{obs}
\label{cl:hyperbola} Let $H_a=\{(x,y)\in \mbb{R}^2 \mid x^2-y^2=a\}$
where $a$ is a non-zero rational representing a hyperbola $H_a$ in
the plane $\mbb{R}^2$. Then the set $H_a(\mbb{Q}) = \{(x,y) \in H_a
\mid x,y \in \mbb{Q}\}$ of rational points is dense in $H_a$.
\end{obs}
\begin{proof}
(of Observation~\ref{cl:hyperbola}) Let $u \in \mbb{R}^{*}$ and set
$x=\frac{u+\frac{a}{u}}{2}, y=\frac{u-\frac{a}{u}}{2}$. We see
immediately that $(x,y) \in H_a(\mbb{Q})$ if $u \in \mbb{Q}^{*}$
since $a$ is rational and the isomorphism $\mbb{R}^{*} \lra H_a$
taking $u$ to $(x,y)$ establishes the Claim~\ref{cl:hyperbola}.
\end{proof}
Continuing with the proof of the Theorem, now consider the part of
the hyperbola $H^{+}_{r^2Sin^2(\gth)}$ corresponding to
$p=x$-coordinate being positive and the following isomorphism
\begin{equation*}
\begin{aligned}
\us{\mbb{R}^{+}}{\ul{u}} &\lra \us{H^{+}_{r^2Sin^2(\gth)}}{\ul{(p=\frac{(u+\frac{r^2Sin^2(\gth)}{u})}{2},s=\frac{(u-\frac{r^2Sin^2(\gth)}{u})}{2})}} \lra \\
&\lra \us{L}{\ul{R=
((rCos(\gth)+s)Cos(\gth),(rCos(\gth)+s)Sin(\gth))}}
\end{aligned}
\end{equation*}
The inverse map being
\begin{equation*}
\begin{aligned}
\us{L}{\ul{R=(x,xTan(\gth))}} &\lra \us{H^{+}_{r^2Sin^2(\gth)}}{\ul{(p=\sqrt{r^2-2xr+\frac{x^2}{Cos^2(\gth)}},s=\frac{x}{Cos(\gth)}-rCos(\gth)})} \lra \\
&\lra \us{\mbb{R}^{+}}{\ul{u=\frac{p+s}{2}}}
\end{aligned}
\end{equation*}
In the maps defined above we note that the co-ordinates of $R$ need
not be rational if $Tan(\gth),Sin(\gth)$ are not rational, however
the distances $PR,QR$ are rational if $Cos(\gth)$ is rational when
we have a rational $u$. From the observation~\ref{cl:hyperbola} we
establish the density of points at a rational distance from $P$ and
$Q$ on the line $L$. Conversely by cosine rule if the distances
$PR,QR$ are rational for some point $R$ on $L$ then by cosine rule
the cosine of the angle between $L$ and $PQ$ is rational. Hence
Theorem~\ref{lemma:DensityLinePQ} follows.
\end{proof}

\section{Rational Points on Families of Quartic and Cubic Curves}

As mentioned in the introduction, our results on quadrilaterals will
proceed by re-expressing the problems in terms of rational points on
families of some quartic and cubic curves. Firstly, we reformulate
the density question for the quadrilateral as a question about
rational points on a two parameter family of quartic curves.

\begin{lemma}
Let $A,O,C$ be three collinear points in the plane with point $O$ on
the line segment $AC$ such that the $3$-set $\{A,O,C\}$ is a
rational set. Let $L$ be a line passing through $O$ such that the
sine and cosine of the angle $\gth$ between and $L$ and $AC$ are
rational. Let $\frac{OC}{AO}=Cot(\gb)$. Then any rational point
$(x,y)$ to the equation
\begin{equation}
\label{eq:Quartic}
\begin{aligned}
y^2     &=  x^4 + p(m,n)x^3 + q(m,n)x^2 + r(m,n)x + 1 \text{ where }\\
p       &=  4(1+n)m = 4Cot(\gth)(1+Cot(\gb))\\
q       &=  4(1+n)^2m^2+4n^2-2 = 4(1+Cot(\gb))^2Cot^2(\gth)+4Cot^2(\gb)-2\\
r       &=  -4(1+n)m = -4Cot(\gth)(1+Cot(\gb))
\end{aligned}
\end{equation}
where $m=Cot(\gth),n=Cot(\gb)$ gives rise to a point $B$ on the line
$L$ such that the distances $AB,CB,OB$ are all rational and
conversely any such point $B$ gives rise to a rational point $(x,y)$
on the quartic curve.
\end{lemma}
\begin{proof}
By rotation and translation if necessary we can assume that the line
$AOC$ is the $x$-axis, $O$ is the origin and $A$ is to the left of
$O$ with coordinates $(-a,0)$ and $C$ is a point to the right of $O$
with coordinates $(c,0)$. Let the line $L$ make an angle $\gth \neq
\frac{\gp}{2}$ with respect to $x$-axis at the origin. The case
$\gth = \frac{\gp}{2}$ can be considered separately.

Consider two families $\mcl{C},\mcl{A}$ of lines passing through the
point $C$ and the point $A$ respectively. Let $m_C$ and $m_A$ denote
the slopes of any two lines one representing a line in the family
$\mcl{C}$ and one representing a line in family $\mcl{A}$
respectively. The equations of the lines are given by

\begin{equation*}
\begin{aligned}
Y &=    m_C(X-c) \ldots \text{ family } \mcl{C}\\
Y &=    m_A(X+a) \ldots \text{ family } \mcl{A}\\
Y &=    Tan(\gth)X \ldots \text{ Line } L
\end{aligned}
\end{equation*}

Any intersection point $B = (X,Y)$ in the plane of two lines one
from family $\mcl{C}$ and one from family $\mcl{A}$ are given by

\begin{equation*}
\begin{aligned}
X   &=  \frac{cm_C+am_A}{m_C-m_A}\\
Y   &=  \frac{(a+c)m_Cm_A}{m_C-m_A}
\end{aligned}
\end{equation*}

The distances $BO,BA,BC$ are given by

\begin{equation*}
\begin{aligned}
BO  &=  \sqrt{\frac{(cm_C+am_A)^2+(a+c)^2m_C^2m_A^2}{(m_C-m_A)^2}}\\
BA  &=  \frac{(a+c)m_C\sqrt{1+m_C^2}}{\sqrt{(m_C-m_A)^2}}\\
BC  &=  \frac{(a+c)m_A\sqrt{1+m_A^2}}{\sqrt{(m_C-m_A)^2}}
\end{aligned}
\end{equation*}

Suppose $B$ lies on the line $L$ then we have

\begin{equation*}
\begin{aligned}
m_A &=  \frac{cm_CTan(\gth)}{(a+c)m_C-aTan(\gth)}\\
BO  &=  \frac{(cm_C+am_A)\sqrt{Sec^2(\gth)}}{\sqrt{(m_C-m_A)^2}}
\end{aligned}
\end{equation*}

So for such a point $B$ on the line $L$, $BO,BA,BC$ are rational if
the following happens.

\begin{equation}
\label{condition:Rational}
\begin{aligned}
&   m_C \text{ is rational.}\\
&   1+m^2_C \text{ is a square of a rational.} \\
&   1+m^2_A \text{ is a square of a rational which is equivalent to }\\
&   ((a+c)m_C-aTan(\gth))^2(1+m^2_A) = c^2m_C^2Tan^2(\gth)+((a+c)m_C-aTan(\gth))^2\\
&   \text{ being a square of a rational.}
\end{aligned}
\end{equation}

Substituting
\begin{equation}
\label{eq:BetaEquation}
\begin{aligned}
\frac{a}{c} &= Tan(\gb)\\
m_C         &= \frac{2Tan(\gga)}{1-Tan^2(\gga)}
\end{aligned}
\end{equation}

the rationality conditions~\ref{condition:Rational} are satisfied if

\begin{equation}
\label{condition:TrigRational}
\begin{aligned}
&   Tan(\gga) \text{ is rational.}\\
&   Tan^4(\gga)+4Cot(\gth)\big(1+Cot(\gb)\big)Tan^3(\gga)+\\
&   \big(4Cot^2(\gb)-2+4\big(1+Cot(\gb)\big)^2Cot^2(\gth)\big)Tan^2(\gga)\\
&   -4Cot(\gth)\big(1+Cot(\gb)\big)Tan(\gga)+1\\
&   \text{ is square of a rational.}
\end{aligned}
\end{equation}

The above rationality condition~\ref{condition:TrigRational} gives
rise to a rational solution to the following
equation~(\ref{eq:Quartic}) and conversely any rational solution
$(x,y)$ to the following equation~(\ref{eq:Quartic}) gives rise to a
value of $Tan(\gga)=x$ and hence the slope $m_C$ with other
rationality conditions~\ref{condition:Rational}
and~\ref{condition:TrigRational} being satisfied.

The case $\gth=\frac{\gp}{2}$ is similar.
\end{proof}

As a second step we transform two parameter family of quartic curves
into a family of cubic curves over rationals.

\begin{lemma}
There exists a $(x,y)-(U,W)$ transformation which transforms the
given two parameter family of quartic curves
\begin{equation}
\label{eq:Quartic}
\begin{aligned}
y^2     &=  x^4 + p(m,n)x^3 + q(m,n)x^2 + r(m,n)x + 1 \text{ where }\\
p       &=  4(1+n)m = 4Cot(\gth)(1+Cot(\gb))\\
q       &=  4(1+n)^2m^2+4n^2-2 = 4(1+Cot(\gb))^2Cot^2(\gth)+4Cot^2(\gb)-2\\
r       &=  -4(1+n)m = -4Cot(\gth)(1+Cot(\gb))
\end{aligned}
\end{equation}
where $m=Cot(\gth),n=Cot(\gb)$

into a two parameter family of cubic curves given by

\begin{equation*}
\begin{aligned}
W^2 &=  U^3+\frac{3p^2-8q}{16}U^2+\frac{3p^4-16p^2q+16q^2+16pr-64}{256}U+\frac{(p^3-4pq+8r)^2}{4096}\\
    &=  U^3+AU^2+BU+C
\end{aligned}
\end{equation*}
where
\begin{equation*}
\begin{aligned}
A   &=  \frac{3p^2-8q}{16}= 1-2Cot^2(\gb)+\big(1+Cot(\gb)\big)^2Cot^2(\gth)\\
    &=  1-2n^2+\big(1+n\big)^2m^2\\
B   &=  \frac{3p^4-16p^2q+16q^2+16pr-64}{256}\\
    &=  -Cot^2(\gb)\big(1+Cot(\gb)\big)\bigg(\big(1-Cot(\gb)\big)+2\big(Cot(\gb)+1\big)Cot^2(\gth)\bigg)\\
    &=  -n^2\big(1+n\big)\bigg(\big(1-n\big)+2\big(n+1\big)m^2\bigg)\\
C   &=  \bigg(\frac{p^3-4pq+8r}{64}\bigg)^2\\
    &=  \big(-Cot(\gb)^2\big(1+Cot(\gb)\big)Cot(\gth)\big)^2\\
    &=  \big(n^4\big(1+n\big)^2m^2\big)
\end{aligned}
\end{equation*}
\end{lemma}
\begin{proof}
Let
\begin{equation*}
P(x)=x^4 + px^3 + qx^2 + rx + 1
\end{equation*}

Notice that the equation~(\ref{eq:Quartic}) can be written as
follows by completing the squares
\begin{equation}
\label{eq:transformquartic}
y^2 =\bigg(x^2+\frac{p}{2}x-\frac{p^2-4q}{8}\bigg)^2 + \bigg(\frac{p(p^2-4q)}{8}+r\bigg)x+
1-\bigg(\frac{p^2-4q}{8}\bigg)^2\\
\end{equation}

Now to get rid of the fourth power in $x$ we substitute $y=U'+\Gs$
where $\Gs=x^2+\frac{p}{2}x-\frac{p^2-4q}{8}$ and also substitute
$x=\frac{V'}{U'}$. So $y$ becomes
$y=(U'+(\frac{V'}{U'})^2+\frac{p}{2}(\frac{V'}{U'})-\frac{p^2-4q}{8})$
and now multiplying by $U'$ on both sides of
equation~(\ref{eq:transformquartic}) we get

\begin{equation*}
\begin{aligned}
U'^3-\bigg(\frac{p^2-4q}{4}\bigg)U'^2+&\bigg(\frac{p^2-4q-8}{8}\bigg)\bigg(\frac{p^2-4q+8}{8}\bigg)U'\\
&   = -2V'^2-pU'V'+\bigg(\frac{p^3-4pq+8r}{8}\bigg)V'
\end{aligned}
\end{equation*}

Replacing $U'$ by $-2U$ and $V'$ by $-2V$ and dividing by $8$ we get

\begin{equation*}
\begin{aligned}
U^3+\bigg(\frac{p^2-4q}{8}\bigg)U^2+&\frac{(p^2-4q-8)(p^2-4q+8)}{256}U\\
&   =V^2+\frac{p}{2}UV+\frac{p^3-4pq+8r}{32}V
\end{aligned}
\end{equation*}

Now to get rid of the $UV-$term we substitute  $W = V + \frac{p}{4}U
+ \frac{p^3-4pq+8r}{64}$ and eliminating $V$  we get the following
two parameter family of cubic curves

\begin{equation*}
\begin{aligned}
W^2 &=  U^3+\frac{3p^2-8q}{16}U^2+\frac{3p^4-16p^2q+16q^2+16pr-64}{256}U+\frac{(p^3-4pq+8r)^2}{4096}\\
    &=  U^3+AU^2+BU+C
\end{aligned}
\end{equation*}
where
\begin{equation*}
\begin{aligned}
A   &= \frac{3p^2-8q}{16}\\
B   &= \frac{3p^4-16p^2q+16q^2+16pr-64}{256}\\
C   &=  \bigg(\frac{p^3-4pq+8r}{64}\bigg)^2\\
\end{aligned}
\end{equation*}
The $(x,y)-(U-V)$ transformation in this case of the quartic
equation~(\ref{eq:Quartic}) is given by
\begin{equation}
\label{eq:xyUVTransfomation}
\begin{aligned}
U       &=  -\frac{1}{2}\bigg(y-x^2-2\big(1+Cot(\gb)\big)Cot(\gth)x+1-2Cot^2(\gb)\bigg)\\
        &=  -\frac{1}{2}\bigg(y-x^2-2(1+n)mx+(1-2n^2)\bigg)\\
V       &=  -\frac{1}{2}\bigg(x\big(y-x^2-2\big(1+Cot(\gb)\big)Cot(\gth)x+1-2Cot^2(\gb)\big)\bigg)\\
        &=  -\frac{1}{2}\bigg(x\big(y-x^2-2(1+n)mx+(1-2n^2)\big)\bigg)\\
x       &=  \frac{V}{U}\\
y       &=  -2U+\bigg(\frac{V}{U}\bigg)^2+2\big(1+Cot(\gb)\big)Cot(\gth)\bigg(\frac{V}{U}\bigg)+2Cot^2(\gb)-1\\
        &=  -2U+\bigg(\frac{V}{U}\bigg)^2+2\big(1+n\big)m\bigg(\frac{V}{U}\bigg)+2n^2-1
\end{aligned}
\end{equation}
and
\begin{equation}
\label{eq:Coefficients}
\begin{aligned}
A(m,n)  &= \frac{3p^2-8q}{16} = 1-2Cot^2(\gb)+\big(1+Cot(\gb)\big)^2Cot^2(\gth)\\
                &=  1-2n^2+\big(1+n\big)^2m^2\\
B(m,n)  &= \frac{3p^4-16p^2q+16q^2+16pr-64}{256}\\
                &=  -Cot^2(\gb)\big(1+Cot(\gb)\big)\bigg(\big(1-Cot(\gb)\big)+2\big(Cot(\gb)+1\big)
                Cot^2(\gth)\bigg)\\
                &=  -n^2\big(1+n\big)\bigg(\big(1-n\big)+2\big(n+1\big)m^2\bigg)\\
C(m,n)  &=  \bigg(\frac{p^3-4pq+8r}{64}\bigg)^2\\
                &=  \big(-Cot(\gb)^2\big(1+Cot(\gb)\big)Cot(\gth)\big)^2\\
                &=  \big(n^4\big(1+n\big)^2m^2\big)
\end{aligned}
\end{equation}
This proves the lemma on transformation from quartics to cubics.
\end{proof}
We observe that the cubic polynomial $Q(U)=U^3+AU^2+BU+C$ has the
following factorization into a linear and a quadratic factor.
\begin{equation}
\label{eq:Factorization}
\begin{aligned}
Q(U)&=U^3+\big(1-2n^2+\big(1+n\big)^2m^2\big)U^2+\\
&\big(-n^2(1+n)\big((1-n)+2(n+1)m^2\big)\big)U+\big(n^4(1+n)^2m^2\big)\\
&=(U-n^2)(U^2+(m^2(1+n)^2-n^2+1)U-m^2n^2(1+n)^2)
\end{aligned}
\end{equation}

The discriminant of $Q(U)$ is given by
\equ{disc(Q(U))=n^4(1+n)^2(1+m^2)((1-n)^2+(1+n)^2m^2)} Now we figure
the points $(m,n) \in \mbb{C}^2$ where the cubic polynomial fails to
have three distinct factors and hence these points $(m,n)\in
\mbb{C}^2$ represent a singular cubic.

We mention the following Lemma~\ref{lemma:RepeatedFactors} without
proof as it is straight forward.
\begin{lemma}[Repeated Factors]
\label{lemma:RepeatedFactors} Let $m,n \in \mbb{C}$. The cubic
polynomial $Q(U)$ has repeated factors if and only if the
discriminant of $Q(U)$ is zero if and only if
\begin{itemize}
\item $n=0,-1$\\
\item $m = \pm \imath$\\
\item For any value of $n \neq -1$, $m= \pm \imath \frac{n-1}{n+1}$
\end{itemize}
\end{lemma}
We quote the lemma below by sketching its proof.
\begin{lemma}
\label{cl:Threeroots} Let $\gt(x) = x^3 + Ax^2 + Bx + C \in
\mbb{R}[x]$ be a cubic polynomial. Then $\gt(x)$ has three distinct
real roots if and only if
\begin{itemize}
\item $A^2-3B > 0$\\
\item $\gt(\frac{-A+\sqrt{A^2-3B}}{3})\gt(\frac{-A-\sqrt{A^2-3B}}{3})$\\
$=  \frac{1}{27}\bigg(-A^2B^2+4B^3+4A^3C-18ABC+27C^2\bigg)<0$\\
\end{itemize}
\end{lemma}
\begin{proof}
Between two distinct reals roots of a polynomial function there is a root of its derivative as a consequence
of Rolle's Theorem. The above lemma follows by observing that there are two real roots of the derivative of
the cubic where the values of the cubic itself have different signs.
\end{proof}
In the case of the polynomial \equ{Q(U) =
(U-n^2)(U^2+(n+1)(m^2(n+1)-n+1)U-m^2n^2(1+n)^2)}

we note that for real values of $m,n$
\begin{itemize}
\item $A^2-3B =\bigg(1-n^2+n^4+2m^2(1+n)^2(1+n^2)+m^4(1+n)^4\bigg)>0$\\
for all $m,n \in \mbb{R}$\\
\item The discriminant of $Q(U)$ is
\equ{\big((1+m^2)n^4(1+n)^2)((-1+n)^2+m^2(1+n)^2\big) > 0} except in
the cases $n=0,-1$ or in the case where $m=0 \text{ and } n=1$ in
which the value is zero.
\end{itemize}

Now we quote the following lemma without
proof.(cf.~\cite{Husemoller} Chapter $0$, Section $7$ Proposition
$7.2 \& 7.3$ Chapter $9$ Section $4$ Theorem $4.3$)
\begin{lemma}
~\\
Let $\gt(x) = x^3 + Ax^2 + Bx + C \in \mbb{R}[x]$ be a cubic
polynomial. Suppose $\gt(x)$ has three distinct real roots. Let $E$
be the elliptic curve defined by $y^2 = x^3 + Ax^2 + Bx + C$. Then
the real locus $E(\mbb{R})$ is isomorphic to $S^1 \times
\mbb{Z}/2\mbb{Z}$.
\end{lemma}
\begin{proof}
This proof is standard as given in the reference.
\end{proof}

So here for $m,n \in \mbb{R}$ the real locus of the elliptic curve
$E(m,n)=\{(U,W)\in \mbb{R}^2 \mid
W^2=U^3+A(m,n)U^2+B(m,n)U+C(m,n)\}$ is isomorphic to $S^1 \times
\mbb{Z}/2\mbb{Z}$ except for the cases $n=0,n=-1$ and the case
$m=0,n=1$ in which the curve is singular.

We quote this lemma without proof as it is straight forward.
\begin{lemma}
\label{cl:JInvariant} Let $r_i \in \mbb{C}, i = 1,2,3$ be distinct
complex numbers. Then the $j$-invariant of the elliptic curve $y^2 =
(x-r_1)(x-r_2)(x-r_3) = x^3+ax^2+bx+c$ is given by
\begin{equation}
\label{eq:NumDenJInvariant}
\begin{aligned}
j-invariant &=  const\frac{(a^2-3b)^3}{a^2b^2-4b^3-4a^3c+18abc-27c^2}\\
            &=  const\frac{(r_1^2+r_2^2+r_3^2-r_1r_2-r_2r_3-r_3r_1)^3}{(r_1-r_2)^2(r_2-r_3)^2(r_3-r_1)^2}
\end{aligned}
\end{equation}
\end{lemma}
In the current case the non-constant $j$-invariant function $j(m,n)
\in \mbb{Q}(m,n)$ given by
\begin{equation*}
\begin{aligned}
&j(m,n)=const*\frac{\big(1-n^2+n^4+2m^2(1+n)^2(1+n^2)+m^4(1+n)^4\big)^3}
{\big((1+m^2)n^4(1+n)^2)((-1+n)^2+m^2(1+n)^2\big)}
\end{aligned}
\end{equation*}
Hence the elliptic variety $V=\{E(m,n)(\mbb{C})\mid (m,n) \in
\mbb{C}\} \lra \{(m,n) \in \mbb{C}^2\}$ has a nonconstant
$j-$invariant as a function of $(m,n)$. In Lemma~\ref{cl:Threeroots}
we have given a condition for a real cubic to have three distinct
roots in terms of the numerator and denominator polynomials of the
$j$-invariant appearing in equation~(\ref{eq:NumDenJInvariant}).

\section{Density of Rational Points on a Family of Elliptic Curves}

Consider the variety $V(m,n)(\mbb{K})$ for a field
$\mbb{K}=\mbb{Q},\mbb{R},\mbb{C}$ defined by the equation
\equa{&V(m,n)(\mbb{K})=\{(U,W)\in \mbb{K}^2 \mid
W^2=U^3+A(m,n)U^2+B(m,n)U+C(m,n)\}\\
&\text{and define }V(\mbb{K})=\us{(m,n)\in
\mbb{K}}{\bigcup}V(m,n)(\mbb{K})} with $A,B,C$ as in the previous
section.

We know that $C(m,n)$ is a square in $\mbb{Q}(m,n)$ from
equation~(\ref{eq:Coefficients}). We establish the density theorem:

\begin{theorem}
\label{theorem:RationalDensity} The set \equa{\mcl{D}_{\mbb{A}} &=
\{kP_1(m,n)\in V(\mbb{Q}) \mid k \in \Z,
(m,n) \in \mbb{A}^2_{\mbb{Q}}-\text{(discriminant locus)-\{m=0\}}\\
&\subs \mbb{A}^2_{\mbb{R}}-\text{(discriminant locus)-\{m=0\}}\}} is
dense in $V(\mbb{R})$ in both Zariski and usual topologies on
$V(\mbb{R})$.
\end{theorem}

Towards the proof, we start by noting that
\begin{equation*}
\begin{aligned}
(W,U)   &= (\pm \frac{(p^3-4pq+8r)}{64},0)\\
            &= (\pm n^2(1+n)m,0)\\
            &= (\pm Cot^2(\gb)(1+Cot(\gb))Cot(\gth),0)
\end{aligned}
\end{equation*}
are two polynomial points on the elliptic variety defined by
$V(m,n)$ over $\mbb{Q}(m,n)$. Since $y=\pm 1, x = 0$ is a solution
to the equation~(\ref{eq:Quartic}) we obtain
\begin{equation*}
\begin{aligned}
(W,U) = (\frac{\pm (r-p)}{8},\frac{-(p^2-4q+8)}{16})\\
(W,U) = (\frac{\pm (r+p)}{8},\frac{-(p^2-4q-8)}{16})
\end{aligned}
\end{equation*}
as the polynomial points on $V$. Consider the polynomial points
\begin{equation*}
\begin{aligned}
P_1(m,n)    &= (\frac{p-r}{8},\frac{-(p^2-4q+8)}{16})\\
                    &= (m(1+n),n^2-1)\\
                    &= (Cot(\gth)(1+Cot(\gb)),Cot^2(\gb)-1)\\
P_2(m,n)    &= (\frac{p^3-4pq+8r}{64},0)\\
                    &= (n^2(1+n)m,0)\\
                    &= (Cot^2(\gb)(1+Cot(\gb))Cot(\gth),0)\\
P_3(m,n)    &= (\frac{p+r}{8},\frac{-(p^2-4q-8)}{16})\\
                    &= (0,n^2)\\
P_4(m,n)    &= (\frac{r-p}{8},\frac{-(p^2-4q+8)}{16})\\
                    &= (-m(1+n),n^2-1)\\
                    &= -P_1(m,n)
\end{aligned}
\end{equation*}

The point $P_3(m,n)$ corresponding to $y=-1,x=0$ is a $2$-torsion
polynomial point on $V$ over $\mbb{Q}(m,n)$. This follows from
equation~(\ref{eq:Factorization}) because $U = n^2$ is a root of the
polynomial $Q(U)$. The points on the elliptic curve corresponding to
the roots of $Q$ and the identity gives rise to the torsion subgroup
$\Z/2\Z \times \Z/2\Z$.

So we prove that the points $P_1,P_4,P_2$ are points of infinite
order in $X(m,n)$ over $\mbb{Q}(m,n)$.

Now we state a very important theorem due to Mazur on torsion orders
of points on an elliptic curve over $\mbb{Q}$.
\begin{theorem}[Mazur's Theorem]
Let $C$ be a non-singular rational cubic curve over $\mbb{K} =
\mbb{C}$ or $\bar{\mbb{Q}}$, and suppose that $C(\mbb{Q})$ contains
a point of finite order $m$. Then either
\begin{equation*}
1 \leq m \leq 10 \text{ or } m = 12
\end{equation*}
\end{theorem}

We observe that $P_1,P_2,P_3$ lie on the line $W + (1+n)m U =
n^2(1+n)m$. So $P_1+P_2 = P_3$ or $P_1+P_2+P_3 = O$. It is enough to
show that the polynomial point $P_1=-P_4$ is of infinite order. Using the computer, 
we can check that $P_1$ does not have polynomial torsion
order $1,2,\ldots,10,12$.

We will show this by a very simple computation using computer that
$kP_1 \neq -P_1$ by explicitly showing that there are points $(m,n)$
where $kP_1(m,n) \neq -P_1(m,n)$ for various $k =
1,2,3,\ldots,10,11$.

The initial point is given by $P_1(m,n) = (m(1+n),n^2-1)$. The
elliptic variety is given by the equation
\begin{equation*}
W^2 - U^3
-(1-2n^2+m^2(1+n)^2)U^2-n^2(n+1)((n-1)-2m^2(n+1))U-m^2n^4(1+n)^2 = 0
\end{equation*}

For a generic point $(m,n) \in \mbb{A}^2_\mbb{K}$ the tangent at
$P_1(m,n)$ meets the elliptic curve $E(m,n)$ at $-2P_1(m,n) =
\bigg(-\frac{(n-1)\big((n-1)^2 + 2
m^2(1+n^2)\big)}{8m^3},\frac{(n-1)^2+4m^2n^2)}{4m^2}\bigg).$

Let $(x_0,y_0),(x_1,y_1)$ lie on the cubic whose equation is given
by $Y^2 = X^3 + \gl X^2 + \gm X + \gn$. Then the line determined by
the points $(x_0,y_0)$ and $(x_1,y_1)$ meets the cubic curve again
at the point $(X,-Y)$ whose values are given by
\begin{equation*}
\begin{aligned}
X &= -\gl + \bigg(\frac{y_1-y_0}{x_1-x_0}\bigg)^2-x_1-x_0\\
Y &= -y_0-(X-x_0)\bigg(\frac{y_1-y_0}{x_1-x_0}\bigg).
\end{aligned}
\end{equation*}
Note that under the elliptic curve addition $(x_0,y_0) + (x_1,y_1) =
(X,Y)$.

Let $k > 0$ be a positive integer. Let $(x,y)=(x(m,n),y(m,n))$ be a
multiple of $P_1$ say $kP_1$. Then the $X$-coordinate $X(m,n)$ and
the $Y$-coordinate $Y(m,n)$ of the multiple $(k+1)P_1$ are given by

\begin{equation}
\label{eq:Iteration}
\begin{aligned}
X[x,y] &= -(1-2n^2+m^2(1+n)^2) + \bigg(\frac{y-m(1+n)}{x-n^2+1}\bigg)^2-n^2+1-x\\
Y[x,y] &=
-m(n+1)-(X[x,y]-n^2+1)\bigg(\frac{y-m(1+n)}{x-n^2+1}\bigg).
\end{aligned}
\end{equation}
Using a computer we can check that the multiples of $P_1$ as
rational functions in $\mbb{Q}(m,n)$. However we will find a
suitable point $(m,n) \in \mbb{A}^2_\mbb{K}$ such that the value
$kP_1(m,n)$ differs from $-P_1(m,n)$ for $k = 1,\ldots,11$ there by
showing that $P_1$ cannot have polynomial torsion
$1,2,\ldots,10,12$.

\begin{lemma}
Let $(m_0,n_0)=(1,2)$. The point $P_1(m_0,n_0)=(3,3)$ on
$E(m_0,n_0)(\mbb{Q})$ is a point of infinite order.
\end{lemma}
\begin{proof}
The two parameter family of cubic curves reduces to the following
elliptic curve at the point $(m_0,n_0)= (1,2)$ given by
\begin{equation*}
X(1,2): W^2 = (-4+U) (-36+6 U+U^2).
\end{equation*}
There are three distinct real values for $U$ where $W$ is zero. They
are given by $U=4,U =3(-1-\sqrt{5}),U=3(-1+\sqrt{5})$.

Consider the point $P_1(m_0,n_0)= (3,3)$ on the elliptic curve
$E(m_0,n_0)(\mbb{Q})$. Using Equation~\ref{eq:Iteration} we compute
the multiples $kP_1(m_0,n_0)$ for
\equ{k \in \{1,2,3,4,5,6,7,8,9,10,11\}.} They are given as follows.\\
\newline
\begin{center}
$\bm{MULTIPLES\ \ OF\ \ P_1(m_0,n_0)}$
\end{center}
\begin{itemize}
\item $P_1(m_0,n_0)  = (W=3,U=3)$,\\
\item $2P_1(m_0,n_0) = (W=\frac{11}{8},U =\frac{17}{4})$,\\
\item $3P_1(m_0,n_0) = (W=-\frac{2091}{125},U =-\frac{189}{25})$,\\
\item $4P_1(m_0,n_0) = (W=-\frac{740943}{85184},U=\frac{11713}{1936})$,\\
\item $5P_1(m_0,n_0) = (W=-\frac{774296133}{1647212741},U=\frac{5104323}{1394761})$,\\
\item $6P_1(m_0,n_0) = (W=\frac{27508807641557}{338608873000},U=\frac{923701649}{48580900})$,\\
\item $7P_1(m_0,n_0) = (W=\frac{3530515935858140877}{285838253719954489},U=-\frac{61622709117}{433923895441})$,\\
\item $8P_1(m_0,n_0) = (W=-\frac{6468618165127547413697}{277205865051779043899904},U=\frac{17006294967389953}{4251429122504256})$,\\
\item $9P_1(m_0,n_0) = (W=-\frac{3133684517758753884882526375341}{268943929702576480749933159625},U=\frac{5746975304186971011}{41665298499035050225})$,\\
\item $10P_1(m_0,n_0)= (W=-\frac{21637825704318812407875118259091920491}{226567277774013103139189522148402968}$,\\
                        $U=\frac{7830395115762668512371857}{371647707091770699565924}$,\\
\item $11P_1(m_0,n_0)= (W=\frac{599341435809994228534143420075705642493847013}{1126066400833062513831952039757204874476841269}$,\\
$U=\frac{3948440455789942950949604475843}{1082370632513496730007602575721})$
\end{itemize}

We observe that $kP_1(m_0,n_0) \neq -P_1(m_0,n_0)$ for all $k \in
\{1,2,\ldots,11\}$ as $P_1(m_0,n_0) \neq -P_1(m_0,n_0)$ and none of
the values for $U$ simplifies to $3$ for $k\geq 2$.
\end{proof}

Now we prove the following lemma.

\begin{lemma}
\label{cl:PolynomialPoint} The point $P_1(m,n)$ is a polynomial
point of infinite order.
\end{lemma}
%\begin{proof}
%We have seen $P_1$ does not have polynomial torsion
%$1,2,\ldots,10,12$ due to a theorem of Mazur on torsion orders of
%rational points on elliptic curves over rationals and observing that
%a polynomial of finite degree in one variable cannot have infinitely
%many roots unless it is a identically zero polynomial and factor
%theorem we conclude this lemma.
%\end{proof}
\begin{proof}
We have seen $P_1$ does not have polynomial torsion
$1,2,\ldots,10,12$. Suppose $kP_1=-P_1$ for some $k>11$. Fix an $n_0
\in \mbb{K}$. Let $m \in \mbb{K}$. We have
$kP_1(m,n_0)=-P_1(m,n_0)$. By Barry Mazur's theorem on torsion
orders of rational points on elliptic curves over rationals,
$lP_1(m,n_0)=-P_1(m,n_0)$ for some  $l \in \{1,2,\ldots, 9,11\}$.
There are finitely many choices for $l$ and infinitely many choices
for $m \in \mbb{K}$. Hence there exists $l \in \{1,2,\ldots,9,11\}$
such that $lP_1(m,n_0)=-P_1(m,n_0)$ for infinitely many $m \in
\mbb{K}$. This means that the equality $lP_1(m,n_0) = -P_1(m,n_0)$
holds as polynomial points in $m$ for the fixed $n_0$. By a similar
argument, again we have for some $l \in \{1,2,\ldots,9,11\}$,
$lP_1(m,n_0)=- P_1(m,n_0)$ for infinitely many $n_0 \in \mbb{K}$ as
polynomial points in $m$. Hence we get $lP_1(m,n) = -P_1(m,n)$ as
polynomial points in $m,n$. However we have verified that $lP_1 \neq
-P_1$ arriving at a contradiction. Now
Lemma~\ref{cl:PolynomialPoint} follows.
\end{proof}

The discriminant locus is defined as the set $\{(m,n) \in \mbb{C}^2
\mid disc(Q(U))=0\}$. Now we state the density lemma for an elliptic
curve.
\begin{lemma}[Density Lemma for an Elliptic Curve]
\label{cl:DensityOfPolynomialPoint} The set

\equa{\mcl{D} &=
\{kP_1(m_0,n_0)\mid k \in \mbb{Z}, P_1(m_0,n_0) \text{ of infinite order
in }E(m_0,n_0)(\mbb{R}),\\ &(m_0,n_0) \in
\mbb{A}^2_{\mbb{Q}}-\text{(discriminant locus)}\}}
is dense in both Zariski and usual topologies in the real locus
$E(m_0,n_0)(\mbb{R})$ of the jacobian elliptic variety $V(\mbb{R})$.
\end{lemma}
\begin{proof}
This is straight forward because any infinite subgroup of the circle
group is dense. However here we need to observe that the following
claim holds.
\begin{claim}
For $(m_0,n_0) \in \mbb{A}^2_{\mbb{R}}-$(the discriminant locus),
the points $\pm P_1(m_0,n_0)$, $\pm P_2(m_0,n_0)$ lie on the oval
component of the real locus. If $P_1(m_0,n_0)$ is of infinite order
in $E(m_0,n_0)(\mbb{R})$ then the subgroup generated by each $P_1$
is a dense subgroup of $S^1 \times \mbb{Z}/2\mbb{Z}$ in both Zariski
and usual topologies on $S^1\times \mbb{Z}/2\mbb{Z}$.
\end{claim}
\begin{proof}
Note that for real $m,n$, the quadratic factor
$(U^2+(n+1)(m^2(n+1)-n+1)U-m^2n^2(1+n)^2)$ of $Q(U)$ evaluated at
$U=0,U=n^2-1$ is non-positive as it is a square polynomial in $n,m$
with a negative sign. So each of the values $U=0,U=n^2-1$ lies
between the roots of the quadratic factor of $Q(U)$. We also have
that the $U$-coordinates of $\pm P_1(m,n),\pm P_2(m,n)$ is less than
or equal to $n^2$ which is the root of the linear factor $U-n^2$ of
$Q(U)$. Hence we observe that the points $\pm P_1(m,n),\pm P_2(m,n)$
lie on the oval.  So under the isomorphism of the real locus
$E(m,n)(\mbb{R})$ to $S^1 \times \mbb{Z}/2\mbb{Z}$, the points $\pm
P_1(m,n),\pm P_2(m,n) \in S^1 \times \{-1\}$ which does not have the
identity element of the group $S^1 \times \mbb{Z}/2\mbb{Z}$. So the
infinite subgroup generated by $P_1$ is a dense subgroup of $S^1
\times \mbb{Z}/2\mbb{Z}$ in both Zariski and usual topologies on
$S^1\times \mbb{Z}/2\mbb{Z}$. This proves the claim.
\end{proof}
Hence the lemma follows.
\end{proof}
\begin{lemma}[Density Bijection Lemma for the Quartics and Cubics]
The set $\{(x,y) \in \mbb{Q}^2, (x,y) \text{ satisfies
equation~}(\ref{eq:Quartic})\}$ is dense in the quartic defined by
the equation~(\ref{eq:Quartic}) in Zariski and usual topologies when
$n \neq 0,-1$ and $m \neq 0$.
\end{lemma}
\begin{proof}
The case $U=0,V=0$ occurs when $y=x^2+\frac p2x - \frac{p^2-4q}{8}$
and if $U=0$ then $W= \pm \frac{(p^3-4pq+8r)}{64} = \mp n^2(1+n)m
\neq 0$.

Using $x-y, U-V$ transformations~(\ref{eq:xyUVTransfomation}) we get
that upon removing the points $(W,U) = \pm P_2(m,n) = (\pm
\frac{(p^3-4pq+8r)}{64},0) = (\pm n^2(1+n)m,0) = (\pm
Cot^2(\gb)(1+Cot(\gb))m,0)$ from the cubic and the point $(x_0,y_0)
= \bigg(\frac{(p^2-4q)^2-64}{8(p^3-4pq+8r)}=-\frac{n-1}{2m}$,
$\frac{(n-1)^2+4m^2n^2}{4m^2}\bigg)$ from the quartic we get a
bijection between the complex solutions of the cubic and quartic.
This bijection restricts to a bijection of their real locus and also
rational locus if $m,n$ are rational. Hence we get the density of
rational solutions $(x,y)$ of the quartic in its real locus
components corresponding under bijection to the real locus
components of the cubic. This completes the proof.
\end{proof}

In order to complete the proof of
Theorem~\ref{theorem:RationalDensity} we give a topological
criterion for the density which is proved in
section~\ref{sec:Appendix}. After giving this criterion for the
density of sets in arbitrary topological spaces we establish
Theorem~\ref{theorem:RationalDensity}.

\section{Main Theorem on Density of Points on a Line at a Rational Distance from Three Collinear Points}
\label{sec:MainTheoremOnDensity} Finally we prove
Theorem~\ref{theorem:DensityLineAOC} on the density of points on a
line which are at a rational distance from three collinear points in
the Euclidean plane from which we deduce rational approximability of
quadrilaterals in the next section.

\begin{proof}
The conditions~\ref{condition:Rational} are satisfied by the
rational points on the quartic. In order to get the required density
to complete the proof of Theorem~\ref{theorem:DensityLineAOC} we use
Mazur's Theorem again in the following way.

First we observe that if for an $m_0 \in \mbb{Q}$,
$kP_1(m,n)=-P_1(m,n)$ for all $n$ then $lP_1(m_0,n) = -P_1(m_0,n)$
for all $n$ for some $l \in \{1,\ldots,9,11\}$. There are finitely
many such $m_0 \in \mbb{Q}$ for $l \in \{1,\ldots,9,11\}$. Let $F_m$
denote the finite set of such elements $m_0 \in \mbb{Q}$. Given any
rational value $m_0$ for $m$ apart from a finite subset $F_m \subs
\mbb{Q}$ with $F_{\measuredangle} = Cot^{-1}(F_m) = \{arcCot(m)\mid
m \in F_m\}$, there exist finitely many $n \in \mbb{Q}$ such that
$kP_1(m_0,n) = -P_1(m_0,n)$ for any integer $k$ because we need to
check only for the finitely many possible torsion order values for
$k$ by Mazur's Theorem. Similarly given any rational value $n_0$ for
$n$ apart from a finite subset $F_n \subs \mbb{Q}$ with $F_{ratio} =
\frac{1}{F_n}$, there exist finitely many $m \in \mbb{Q}$ such that
$kP_1(m,n_0) = -P_1(m_0,n)$ for any integer $k$. Hence for a given
rational value $m_0 \nin F_m$ for $m$, $P_1(m_0,n)$ is a point of
infinite order for all but finitely many $n$ and for a given
rational value $n_0 \nin F_n$ for $n$, $P_1(m,n_0)$ is a point of
infinite order for all but finitely many $m$.

This proves the Theorem.
\end{proof}

\section{Rational Approximability of Quadrilaterals}

\begin{theorem}
\label{th:RationalQuadrilateral} Let $X =\{A,B,C,D\}$ represent four
vertices of a quadrilateral in the Euclidean plane. Then given $\gep
> 0$ there exists a rational quadrilateral $X_{\gep}$ in the plane
such that $D(X,X_{\gep}) < \gep$.
\end{theorem}
\begin{proof}
We prove this theorem in a few steps. We rename the vertices of the
quadrilateral such that if $\Gd ABD$ is the triangle formed by three
out of four vertices such that if the quadrilateral is concave then
the point $C$ is in the interior of $\Gd ABD$ or on the $\Gd ABD$
and if the quadrilateral is convex then it lies in the exterior.

We give a slightly elaborate proof in the convex quadrilateral case
and give a less elaborate but similar proof in the concave case.

So consider the case of a convex quadrilateral $\square ABCD$.

$\bm{Step: 1}$

Let the diagonals $AC$,$BD$ meet at a point $O$. We assume that
$\measuredangle AOB$ is greater than or equal to $\frac{\gp}{2}$
(i.e. just right or obtuse) by renaming the vertices $\{A,B,C,D\}$
of the quadrilateral so that in the $\Gd AOB$, the side $AB$ is the
largest side. Now we approximate $\Gd AOB$ using
Theorem~\ref{Th:RationalTriangle} by a rational triangle with
rational area $\Gd A'O B'$ such that
$D(\{A,O,B\},\{A',O,B'\})<\gd_1$. Given a $\gd_2>0$, we note that by
suitably choosing smaller $\gd_1$ if necessary we can assume that
$D(\{A,O,B\},\{A',O,B'\})$, $<\gd_2$ and $d(C,\text{ Line
}A'O)<\gd_2$. Since $\Gd A'OB'$ is rational with rational area we
have that both the sine and the cosine of the angles $\measuredangle
A'OB',\measuredangle OA'B',\measuredangle OB'A'$ are rational.

While obtaining an approximant $\Gd A'OB'$ for the triangle $\Gd
AOB$ by using Theorem~\ref{Th:RationalTriangle}, we make sure that
the mentioned angles in Theorem~\ref{Th:RationalTriangle}
$\measuredangle OA'B'=\frac {\gp}{2}-\ga_1,\measuredangle
OB'A'=\frac{\gp}{2}-\gb_1$ are so chosen that
$\ga_1,\gb_1,\ga_1+\gb_1,\frac{\gp}{2}-\ga_1,\frac{\gp}{2}-\gb_1,\gp-(\ga_1+\gb_1)$
are not in $F_{\measuredangle} \cup (\gp-F_{\measuredangle})$ which
is a finite set where this finite set $F_{\measuredangle}$ arises in
Theorem~\ref{theorem:DensityLineAOC}. Hence $\measuredangle A'OB'
\nin F_{\measuredangle} \cup (\gp-F_{\measuredangle})$ where
$F_{\measuredangle}$ is the set described in
Theorem~\ref{theorem:DensityLineAOC}.

$\bm{Step: 2}$

Using Lemma~\ref{lemma:DensityLinePQ} again by suitably choosing
$\gd_1,\gd_2$ we can find a point $C'$ on the line $A'O$ such that
$D(\{A,B,C\},\{A',B',C'\})<\gd_3$. We also make sure the choice of
$C'$ on the line $A'O$ is such that the ratio $\frac{A'O}{OC'}$ is
not one of those finitely many ratios in $F_{ratio}$ corresponding
to the angle $\measuredangle A'OB'$ as described in the
Theorem~\ref{theorem:DensityLineAOC} which may not give density as
per Theorem~\ref{theorem:DensityLineAOC}.

Now we also have that $\Gd B'OC'$ is a rational triangle because the
$3$-set $\{B',O,C'\}$ is rational and $Sin(\measuredangle B'OC')$ is
rational which follows because $Sin(\measuredangle A'OB')$ is
rational and hence the $\Gd A'B'C'$ is a rational triangle with
rational area.

$\bm{Step: 3}$

Again by suitably choosing $\gd_1,\gd_2,\gd_3 < \gep$ we can assume
that $d(D,\text{ Line }B'O) < \gep$. Since the ratio
$\frac{A'O}{OC'} \nin  F_{ratio}$ using
Theorem~\ref{theorem:DensityLineAOC} because of density we can find
a point $D'$ on the line $B'O$ such that the set
$D(\{A,B,C,D\},\{A',B',C',D'\}) < \gep$ and the set
$\{A',O',C',D'\}$ is a rational set. So the triangles \equ{\Gd
A'OD',\Gd B'OD',\Gd C'OD',\Gd A'OD'} are rational triangles with
rational area. Hence by taking $X_{\gep} = \{A',B',C',D'\}$ we have
that
\begin{itemize}
\item $X_{\gep}$ is a rational set.
\item Area of the quadrilateral $\square A'B'C'D'$ is rational.
\item $D(X,X_{\gep}) < \gep$.
\end{itemize}
In the case when $B,C,D$ are collinear with $C$ in between $B,D$ and
$A$ is outside the line of $B,C,D$ one of the angles $\measuredangle
BCA,\measuredangle DCA$ is greater than or equal to $\frac{\gp}{2}$
(i.e. just right or obtuse). So we use
Theorems~\ref{Th:RationalTriangle},~\ref{theorem:DensityLineAOC} to
find a rational $4-$set $\{A',C,B',D'\}$ with the point $C$ in
common with $X$ which gives the density.

If all four points $\{A,B,C,D\}$ are collinear then we use the
density of rationals in reals for the conclusion.

If the $4-$ set $\{A,B,C,D\}$ form a concave quadrilateral then we
let $C$ be in the interior of the $\Gd ABD$ and let $AC$ intersect
$BD$ at $O'$. Assume $\measuredangle AO'B$ is obtuse or just right
without loss of generality so that $\measuredangle ACB$ is also
obtuse and we apply Theorem~\ref{Th:RationalTriangle} with vertex
$C$ opposite to the largest side $AB$ so that $C$ is in the
approximating rational $3$-set. Then we find $O$ close to $O'$ by
density and now in the proof of Theorem~\ref{theorem:DensityLineAOC}
we run through the argument with the coordinates of the point
$C=(-c,0)$ with $a>c>0$. Here again we conclude density similar to
the convex case by replacing $n$ with $-n$ and $n$ values with $-n$
values. It is the same argument about choosing proper angles for
density purposes not in the associated finite set
$F_{\measuredangle}$ and the distances also such that the ratio is
not from the finite set $F_{ratio}$ one we fix an angle. Hence
Theorem~\ref{th:RationalQuadrilateral} is proved.
\end{proof}
\section{Appendix}
\label{sec:Appendix}
\begin{lemma}
Let $X,Y$ be topological spaces and $f:X \lra Y $. Suppose $f$ is a
closed surjective map and if $A \subs X$ is saturated then $\bar{A}$
is saturated. Let $B\subs Y$ be such  that $\bar{B}=Y$. Let
$A=f^{-1}(B)$. Then $\ol{A}=X$.
\end{lemma}
\begin{proof}
Let $C \subs X$ be any closed set containing $A \subs X$. Then $f(C)
\sups f(A)=B$ and it is closed. So $f(C)=Y$. If $C$ is saturated
then $C=X$. Consider $C=\bar{A}$. By the hypothesis of this lemma
$\bar{A}$ is saturated. Hence $\ol{A}=X$.
\end{proof}
\begin{example}
\begin{enumerate}
\item Let $X=[0,1] \times [0,1]$ and $Y=[0,1]$. Let $f=\gp$ be the first projection.
\item Let $X=\bigg([0,1] \times [0,1]\bigg) \sqcup [2,3]$. Let $f=\gp$ be the first projection on
$[0,1] \times [0,1]$ and $f(t)=1$ if $t \in [2,3]$.
\item Let $X=\bigg([0,1] \times [0,1]\bigg) \cup [1,2] \times \{0\}$. Let $f=\gp$ be the first projection on
$[0,1] \times [0,1]$ and $f(t)=1$ if $t \in [1,2]  \times \{0\}$.
\item Let $X=[0,1],Y=\{0,1\}$ with topology $\{\es,\{0\},\{0,1\}\}$. Define $f:X \lra Y$ as
$f(t)=0$ for $0 \leq t < \frac 12$ and $f(t)=1$ for $\frac 12 \leq t
\leq 1$. Then $f$ is a surjective continuous map. Take $B=\{0\}$.
Then $A=[0,\frac 12)$ which is not dense in $X$.
\end{enumerate}
The cases $2,3$ give examples where we have even if $f$ is a closed
surjective map, if $A$ is saturated then $\ol{A}$ need not be
saturated. Take $A=[0,1) \times [0,1]$. Here in one case $X$ is
disconnected and in other case $X$ is connected.
\end{example}

\begin{lemma}[Local Product Structure Lemma]
\label{lemma:LocalProductStructureLemma}
Let $X,Y$ be topological spaces. Let $f:X \lra Y$ be topological
spaces. Suppose $X$ has the local product structure property with
respect to $f$ on a dense subset $Z \subs X$. Let $B\subs Y$ be such
that $\bar{B}=Y$. Let $A=f^{-1}(B)$. Then $\ol{A}=X$.
\end{lemma}
\begin{proof}
Let $x \in Z$. Let $(x \in O \subs X,U \subs F_{x},y=f(x) \in V
\subs Y,O \cong_{\psi} U \times V)$ be a local product structure at
$x$. Let $\psi(x)=(u,y) \in U \times V$.  Since $V$ is open we have
$V \cap B$ is dense in $V$. So $\{u\} \times (V \cap B)$ is dense in
$\{u\} \times V$ and we also have $U \times (V \cap B)$ is dense in
$U \times V$ and $\psi^{-1}(U \times (V \cap B)) = A \cap O =
f^{-1}(V \cap B) \cap O$. Hence $x \in O \subs \ol{A}$ which implies
$Z \subs \ol{A}$. So $\ol{A}=X$ and the lemma follows.
\end{proof}
\subsection{Two Applications}
This above lemma can be applied in many instances. In this subsection
below we give two applications.
\begin{lemma}[First Application: Existence of Local Product Structure on Dense Set]
\label{lemma:FAELPSDS}
Let $X=\{(x_1,x_2,\ldots,x_n,y) \in \mbb{R}^n \mid
y^2-P[x_1,x_2,\ldots,x_n]=0$ where $P[x_1, x_2,$
$\ldots,x_n] \in \mbb{R}[x_1,x_2,\ldots,x_{n-1}][x_n]$ a monic polynomial in $x_n$
with coefficients in $\mbb{R}[x_1,x_2,\ldots,x_{n-1}]$. Let the map
$f: X \lra \mbb{R}^{n-1}$ given by $f(x_1,x_2,\ldots,x_n,y)$
$=(x_1,x_2,\ldots,x_{n-1})$ be a sujective map. Then $X$ has the
local product structure property with respect to $f$ on a dense
subset of $X$.
\end{lemma}
\begin{proof}
Consider the set \equ{Z= X\bs \{(x_1,x_2,x_3,\ldots,x_{n},y) \in
\mbb{R}^{n+1} \mid P[x_1,x_2,x_3,\ldots,x_{n}] \neq 0\}.}

Then we prove that $Z\subs X$ has local product structure property
at every point. For this purpose let $(x_1^0,x_2^0,\ldots,x_n^0,y^0)
\in Z$. Since $P[x_1^0,x_2^0,x_3^0,\ldots,x_{n}^0]\neq 0$ we have
$y^0 \neq 0$. Hence there exist an open set $V \subs \mbb{R}^{n-1}$
such that $(x_1^0,x_2^0,x_3^0,\ldots,x_{n-1}^0) \in V$ and $\gep >0$
such that $P[x_1,x_2,x_3,\ldots,x_{n}] \neq 0$ for all
$(x_1,x_2,x_3,\ldots,x_{n}) \in V \times (x_n^0-\gep,x_n^0+\gep)$.
Choosing the space
$F_{(x_1^0,x_2^0,x_3^0,\ldots,x_{n}^0,y^0)}=U=(x_n^0-\gep,x_n^0+\gep),
O=\{(x_1,x_2,x_3,\ldots,x_{n},y) \in X \mid
(x_1,x_2,x_3,\ldots,x_{n}) \in V \times U, sign(y)=sign(y_0)\}$ and
we define a map $\psi: O \lra V \times U$ given by
$\psi(x_1,x_2,x_3,\ldots,x_{n},y)=((x_1,x_2,x_3,\ldots,x_{n-1}),x_n)$.
Clearly $O$ is open as the sign condition can be treated as open
condition over the reals. Now the lemma follows.
\end{proof}
\begin{lemma}[First Application: Density]
\label{lemma:FirstDensityLemma}
Let $X=\{(x_1,x_2,\ldots,x_n,y) \in \mbb{R}^n \mid
y^2-P[x_1,x_2,\ldots,x_n]=0$ where $P[x_1, x_2,\ldots,x_n] \in
\mbb{R}[x_1,x_2,\ldots,x_{n-1}][x_n]$ a monic polynomial in $x_n$
with coefficients in $\mbb{R}[x_1,x_2,\ldots,x_{n-1}]$. Let the map
$f: X \lra \mbb{R}^{n-1}$ given by
$f(x_1,x_2,\ldots,x_n,y)=(x_1,x_2,\ldots,x_{n-1})$ be a sujective
map. Let $B \subs \mbb{R}^{n-1}$ be a dense set. Then $A=f^{-1}(B)$
is dense in $X$.
\end{lemma}
\begin{proof}
Using the previous two lemmas~\ref{lemma:LocalProductStructureLemma},~\ref{lemma:FAELPSDS} this lemma follows. This also proves Theorem~\ref{theorem:RationalDensity}.
\end{proof}
\begin{lemma}[Second Application: Existence of Local Product Structure on Dense Set]
\label{lemma:SAELPSDS}
Let $X=\{(x_1,x_2,\ldots,x_n,y_1,y_2,\ldots,y_m) \in
\mbb{R}^{n+m} \mid F_i[x_1,x_2,\ldots,x_n,$
$y_1,y_2,\ldots,y_m]=0$ for $1 \leq i \leq m$ where
$F_i$ is a polynomial function.$\}$
Suppose $f:X \lra \mbb{R}^n$ given by $(x_1,x_2,\ldots,x_n,y_1,y_2,\ldots,y_m) \lra (x_1,x_2,\ldots,x_n)$ is a surjective map. Then $f$ has the local product structure property.
\end{lemma}
\begin{proof}
Let $Z=\{(x_1,x_2,\ldots,x_n,y_1,y_2,\ldots,y_m)\mid
det\big((\frac{\partial F_i}{\partial y_j})^{m,m}_{i=1,j=1}\big)\}
\neq 0$. Let $(x^0_1,x^0_2,\ldots,x^0_n,y^0_1,y^0_2,\ldots,y^0_m)
\in Z$ then there exists open sets $V \subs \mbb{R}^n,O \subs X$ and
map $\gf: V \lra \mbb{R}^m$ such that
$F_i\big(x_1,x_2,\ldots,x_n,(y_1,y_2,\ldots,y_m)=\gf(x_1,x_2,\ldots,x_n)\big)=0$
for all $(x_1,x_2,\ldots,x_n)\in V$ for all $1 \leq i \leq m$  and
$graph(\gf) = O\subs X$ by implicit function theorem. So
$O=\{(x_1,x_2,\ldots,x_n,y_1,y_2,\ldots,y_m) \mid
(x_1,x_2,\ldots,x_n) \in V,(y_1,
y_2,\ldots,y_m)=\gf(x_1,x_2,\ldots,x_n)\}$. Now take the space
$F_{(x^0_1,x^0_2,\ldots,x^0_n,y^0_1,y^0_2,\ldots,y^0_m)}=U=\{a\}$ to
be a singleton topological space and define a map $\psi:O \cong V
\times U$ as $\psi: (x_1,x_2,\ldots,x_n,y_1,y_2,\ldots,y_m)
=((x_1,x_2,\ldots,x_n),a)$. Then $\psi$ is a homeomorphism. We see
that the following diagram commutes. \equ{\big(O \us{f}{\lra} V\big)
= \big(O \us{\psi}{\lra} V \times U \us{\gp_1}{\lra} V\big)} This
proves the lemma.
\end{proof}
\begin{lemma}[Second Application: Density]
\label{lemma:SecondDensityLemma}
Let $X=\{(x_1,x_2,\ldots,x_n,y_1,y_2,$
 $\ldots,y_m) \in \mbb{R}^{n+m} \mid F_i[x_1,x_2,\ldots,x_n,$ $y_1,y_2,\ldots,y_m]=0$ for $1 \leq i \leq m$
where  $F_i$ is a polynomial function $\}.$ Suppose
$f:X \lra \mbb{R}^n$ given by
$(x_1,x_2,\ldots,x_n,y_1,y_2,\ldots,y_m)$
$\lra (x_1,x_2,\ldots,x_n)$ is a surjective map. Let $B \subs \mbb{R}^{n}$ be a dense set. Then
$A=f^{-1}(B)$ is dense in $X$.
\end{lemma}
\begin{proof}
Using the lemma~\ref{lemma:LocalProductStructureLemma} and the previous lemma~\ref{lemma:SAELPSDS} this lemma follows.
\end{proof}
\begin{lemma}
\label{lemma:RepeatedPrinciple}
Let $X_0 \us{f_1}{\lra} X_1 \us{f_2}{\lra} X_2 \us{f_3}{\lra} \ldots \us{f_n}{\lra} X_n$ be a sequence of surjective continuous maps of topological spaces such that the local product structure property is satisfied on a dense set $Z_i$ in $X_i$ with respect to
the map $f_{i+1}$ for $i=0,\ldots,n-1$. Then if $B \subs X_n$ is
dense then the preimage of $B$ in each $X_i$ is dense in $X_i$ for
all $0 \leq i \leq n-1$.
\end{lemma}
\begin{proof}
By a repeated application of the same principle this lemma follows.
\end{proof}
Now we prove Theorem~\ref{Theorem:DensityFibrations}.
\begin{proof}
Using the previous lemma~\ref{lemma:RepeatedPrinciple} and the observation that in the
closure of fibrewise dense set the entire fibre is there and hence upon its
closure we get the whole space.
\end{proof}
\section{Acknowledgments}
I would like to thank Prof. C.R. Pranesachar, Indian Institute of
Science, Bangalore, Prof. Jaya Iyer, The Insititute of Mathematical
Sciences, Chennai and Prof. B. Sury, Indian Statistical Institute,
Bangalore for their motivation, suggestions of revisions during the
writing of the document. I would like to dedicate this article to my
sister C.P. Aparna and my mother C.P. Satyavathi.
\renewcommand{\bibname}{}

\end{document}